\documentclass[1p]{elsarticle_modified}
\usepackage[colorlinks]{hyperref}
\usepackage{abbrmath_seonhwa} 
\usepackage{amsfonts}
\usepackage{amssymb}
\usepackage{amsmath}
\usepackage{amsthm}
\usepackage{amsbsy}
\usepackage{kotex}
\usepackage{caption,subcaption}
\usepackage{color}
\usepackage{graphicx}
\usepackage{xcolor} 
\usepackage{float}
\usepackage{setspace}

\def\psl{\operatorname{\textup{PSL}}(2,\Cbb)}

\theoremstyle{definition}
\newtheorem{thm}{Theorem}[section]
\newtheorem{prop}[thm]{Proposition}

\newtheorem{defn}[thm]{Definition}
\newtheorem{exam}[thm]{Example}

\begin{document}

\begin{frontmatter}

\title{Partially abelian representations of knot groups}

\author{Yunhi Cho} 
\address{Department of Mathematics, University of Seoul, Seoul, Korea}
\fntext[y_cho]{The first author was supported by the 2014 sabbatical year research grant of the University of Seoul.}
\ead{yhcho@uos.ac.kr}

%

\author{Seokbeom Yoon}
\address{Department of Mathematical Sciences, Seoul National University, Seoul 08826,  Korea}
\fntext[s_yoon]{The second author was supported by Basic Science Research Program through the NRF of Korea funded by the Ministry of Education (2013H1A2A1033354).}
\ead{sbyoon15@snu.ac.kr}

\begin{abstract}
	A knot complement admits a pseudo-hyperbolic structure by solving Thurston's gluing equations for an octahedral decomposition. It is known that a solution to these equations can be described in terms of region variables, also called $w$-variables. In this paper, we consider the case when pinched octahedra appear as a boundary parabolic solution in the decomposition. A $w$-solution with pinched octahedra induces a solution for a new knot obtained by changing the crossing or inserting a tangle at the pinched place. We discuss this phenomenon with corresponding holonomy representations and give some examples including ones obtained from connected sum.
\end{abstract}
\begin{keyword}
	knot diagram change, boundary parabolic representation.
    \MSC[2010] 57M25 
\end{keyword}

\end{frontmatter}

\section{Introduction}\label{sec:introduction}

For a knot diagram $D$ of a knot $K$ in $S^3$, D.Thurston \cite{thurston_1999} introduced a way to decompose $M=S^3 \setminus (K \cup \{ \textrm{two points}\})$ into ideal octahedra by placing an octahedron at each crossing and then identifying their faces appropriately along the knot diagram. One can obtain an ideal triangulation $\Tcal_D$ of $M$ by dividing each octahedron into ideal tetrahedra. Then we can give a ``pseudo-hyperbolic structure" on $M$ through this ideal triangulation by solving Thurston's gluing equations for $\Tcal_D$ which require the product of cross-ratios (or shape parameters) around each edge of $\Tcal_{D}$ to be $1$. Since the cross-ratios determine the shapes of each ideal hyperbolic octahedron and vice versa, these hyperbolic octahedra, giving a pseudo-hyperbolic structure on $M$, will be called a solution. Even though the gluing equations only guarantee that the sum of dihedral angles around each edge is a multiple of $2\pi$, not $2\pi$, one still can consider a (pseudo-) developing map of $M$ with a holonomy representation as W.Thurston did in \cite{thurston_geometry_1977}, whenever a solution is given. 

An octahedral decomposition has been used by several authors very successfully in conjunction with the volume conjecture. Yokota \cite{yokota_potential_2002} used a 4-term triangulation $\Tcal_{4D}$ of $M$ motivated by the optimistic limit of the Kashaev invariant presenting the gluing equations as derivatives of a potential function. In a similar manner, Cho and Murakami \cite{CM_2013} suggested a 5-term triangulation $\Tcal_{5D}$ of $M$ applying the optimistic limit to the colored Jones polynomial formulation of the state sum of quantum invariant. They present the gluing equations for $\Tcal_{5D}$ in terms of region variables, also called $w$-variables, which are non zero complex valued variables assigned to each region of a diagram $D$. The 5-term triangulation $\Tcal_{5D}$ has a nice property that any non trivial boundary parabolic representation of a knot group can be derived from a solution to the gluing equations for $\Tcal_{5D}$ as a holonomy  \cite{cho_optimistic_2016}. On the other hand, the 4-term triangulation $\Tcal_{4D}$ does not have such property since the octahedron at a crossing in $\Tcal_{4D}$ can not be pinched, i.e., the top and bottom vertices of the octahedron can not coincide, while the octahedron in $\Tcal_{5D}$ can. 

	\[
	\vcenter{\hbox{
\begingroup%
  \makeatletter%
  \providecommand\color[2][]{%
    \errmessage{(Inkscape) Color is used for the text in Inkscape, but the package 'color.sty' is not loaded}%
    \renewcommand\color[2][]{}%
  }%
  \providecommand\transparent[1]{%
    \errmessage{(Inkscape) Transparency is used (non-zero) for the text in Inkscape, but the package 'transparent.sty' is not loaded}%
    \renewcommand\transparent[1]{}%
  }%
  \providecommand\rotatebox[2]{#2}%
  \ifx\svgwidth\undefined%
    \setlength{\unitlength}{283.46456693bp}%
    \ifx\svgscale\undefined%
      \relax%
    \else%
      \setlength{\unitlength}{\unitlength * \real{\svgscale}}%
    \fi%
  \else%
    \setlength{\unitlength}{\svgwidth}%
  \fi%
  \global\let\svgwidth\undefined%
  \global\let\svgscale\undefined%
  \makeatother%
  \begin{picture}(1,0.35)%
    \put(0,0){\includegraphics[width=\unitlength,page=1]{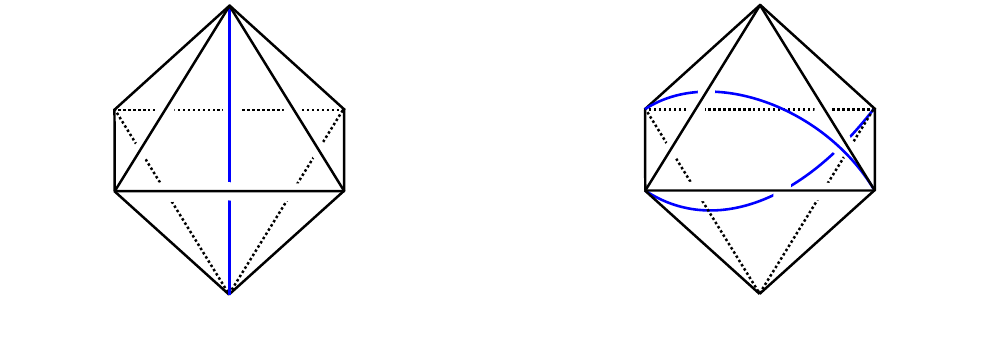}}%
    \put(0.0292494,0.00481801){\color[rgb]{0,0,0}\makebox(0,0)[lb]{\smash{(a) 4-term triangulation}}}%
    \put(0.57111609,0.00764023){\color[rgb]{0,0,0}\makebox(0,0)[lb]{\smash{(b) 5-term triangulation}}}%
  \end{picture}%
\endgroup%
}}
	\]
    
Each solution of the gluing equations gives rise to a holonomy representation of a knot group and among them only boundary parabolic ones will be considered in this paper.
We observed that some interesting phenomena arise when pinched octahedra appear in a solution. We first suggest a notion of R-related diagrams as follows. Let a solution to the gluing equations for $\Tcal_{5D}$ have pinched octahedra. Then it also satisfies the gluing equations for $\Tcal_{5D'}$ where $D'$ is a new diagram obtained from $D$ by changing a crossing at which a pinched octahedron is assigned (Theorem \ref{thm:main}). We say two such diagrams $D$ and $D'$, having a ``common $w$-solution", are R-related. Here `R' stands for `representation' meaning that both knots $K$ and $K'$, represented by $D$ and $D'$ respectively, have representations of knot groups with the same image group in $\psl$. These representations is called partially abelian representations where meaning of  ``partially abelian'' will be explained in the following section. We also show that whenever a pinched solution arises, we can replace the crossing, where the pinch occurs, by rational tangles with a relatively easy change of $w$-solutions (Theorem \ref{thm:main2}). This shows that we can construct lots of ``bigger" knots having the same representations and the complex volume as the one we started with. In the last section, we describe how we can find examples of R-related diagrams through the connected sum.
 \section*{Acknowledgment}
 We thank Hyuk Kim and Seonhwa Kim for helpful comments and suggestions.

\section{Region variables and pinched octahedra}\label{secprelim}
	\subsection{Region variables}	
	Let $D$ be a knot diagram of a knot $K$ with $N$ crossings. We denote the crossings of $D$ by $c_1,\cdots,c_N$ and the regions of $D$ by $r_1,\cdots,r_{N+2}$. Let $\Ocal_D$ be Thurston's octahedral decomposition of $M=S^3\setminus(K\cup\{ \textrm{two points}\})$ with respect to $D$. We denote the ideal octahedron of $\Ocal_D$ at a crossing $c_k$ by $o_k$. We divide each octahedron $o_k$ into five tetrahedra by adding two edges as in Figure \ref{fig:octa_region} and call the resulting ideal triangulation of $M$ \emph{the five-term triangulation} $\Tcal_{5D}$. Considering the octahedra $o_1,\cdots,o_N$ to be hyperbolic, Cho and Murakami \cite{CM_2013} suggested region variables as a way to describe shape of the hyperbolic octahedra. A region variable $w_j$ is a non zero complex valued variable  assigned to a region $r_j$ of $D$ where the ratio of adjacent region variables around $c_k$ becomes the shape parameter of a tetrahedron in $o_k$ as in Figure~\ref{fig:octa_region}.
	\begin{figure}[H]
		\centering
		\scalebox{1}{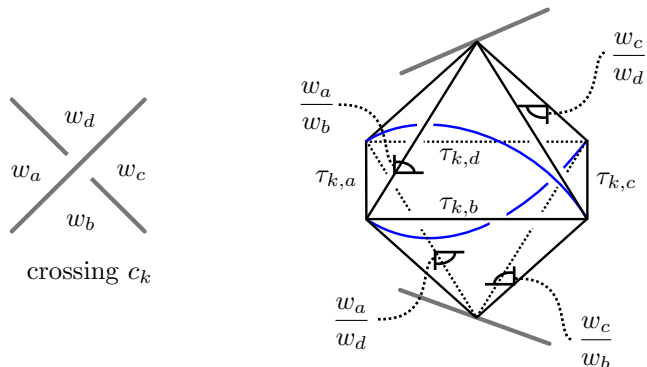}
		\caption{The 5-term triangulation and region variables.}
		\label{fig:octa_region}
	\end{figure}
	It turns out that the hyperbolic ideal octahedra $o_1,\cdots,o_N$ whose shapes are determined by $w$-variables as above automatically satisfy the gluing equation for every edge of $\Tcal_{5D}$ except for the edges corresponding to the regions of $D$; see Section 4.3 of \cite{KKY} for details. 
	The gluing equations for these edges are
	$$\prod_{\textrm{corner crossing } c_k \textrm{ of }r_j}\tau_{k,j}=1$$ for each region $r_j$ of $D$ where the product is over all corner crossings of a region $r_j$ and $\tau_{k,j}$ is the cross-ratio at the side edge of $o_k$ corresponding to $r_j$. (See Figure~\ref{fig:octa_region}.) Since both the ratios of $w$-variables and $\tau$'s are cross-ratios at the edges of $o_k$, from the general relation of these cross-ratios of an octahedron, one can compute $\tau_{k,*}$'s in terms of region variables as follows :
	\begin{equation}\label{eqn:tau_w_a}
		\begin{array}{cc}
		\left\{
		\begin{array}{rl}
		\tau_{k,a} =& \dfrac{ w_b w_d -w_a w_c}{( w_a- w_b) (w_a - w_d)} \\[13pt]
		\tau_{k,b} =& \dfrac{(w_b-w_c) (w_b-w_a)}{w_a w_c-w_b w_d} \\[13pt]
		\tau_{k,c} =& \dfrac{ w_b w_d -w_a w_c}{( w_c-  w_d) (w_c -  w_b)} \\[13pt]
		\tau_{k,d} =& \dfrac{(w_d-w_a) (w_d-w_c)}{w_a w_c-w_b w_d}
		\end{array}
		\right. & \textrm{ for a crossing $c_k$ as in Figure \ref{fig:octa_region}}
		\end{array}
	\end{equation} 
	\begin{defn} A \emph{region variable} $w_j$ is a non-zero complex valued variable assigned to a region $r_j$ for $1 \leq j \leq N+2$. A $(N+2)$-tuple of region variables $w=(w_1,\cdots,w_{N+2})$ is a \emph{boundary parabolic solution} (to Thurston's gluing equations for $\Tcal_{5D}$) if it satisfies \\
	(a) (gluing equation) 
	\begin{equation}\label{eqn:gluing}
	\prod_{\textrm{corner crossing } c_k \textrm{ of }r_j}\tau_{k,j}=1
	\end{equation} for every region $r_j$ of $D$ \\
	(b) (non-degeneracy condition) $w_a w_c-w_b w_d \neq 0$ at each crossing as in Figure~\ref{fig:octa_region} and every pair of adjacent region variables is distinct.
	\end{defn}
	 Here the non-degeneracy condition (b) holds if and only if every ideal hyperbolic tetrahedron of $\Tcal_{5D}$ is non-degenerate. (We refer Section 4.3 of \cite{KKY} for the details in this subsection.)
\subsection{Pinched octahedra}
	Let region variables $w$ be a boundary parabolic solution and let $o_k$ be the corresponding hyperbolic ideal octahedron of $\Ocal_D$ at a crossing $c_k$ whose cross-ratios are determined by region variables $w$. One can construct a pseudo-developing map of $M=S^3\setminus(K \cup \{ \textrm{two points}\})$ by placing the octahedra $o_1,\cdots,o_N$ consecutively in $\Hbb^3$ in the fashion arranged in the universal cover $\widetilde{M}$. Then one can  obtain a holonomy representation $\rho : \pi_1(M)\rightarrow \psl$ of the knot group by the rigidity of a developing map. Thus the boundary parabolic solution $w$ gives a boundary parabolic holonomy representation.
	
	In \cite{KKY}, they observed that an octahedron $o_k$ may be \emph{pinched}, i.e., the top and bottom vertices of $o_k$ coincide. 
	
 	\begin{prop}\label{prop:pinch}	Let  $m_k$ and $\widehat{m}_k$ be Wirtinger generators winding the over-arc and the incoming under-arc of $c_k$, respectively. Then the following are equivalent.\\
	 	(a) The hyperbolic octahedron $o_k$ is pinched.\\
		(b) $w_a-w_b+w_c-w_d=0$ for Figure \ref{fig:octa_region}. \\
		(c) $\tau_{k,j}$ = 1 for some region $r_j$ adjacent to $c_k$.\\
		(d) $\tau_{k,j}$ = 1 for every region $r_j$ adjacent to $c_k$.\\
		(e) $\rho(m_k)$ and $\rho(\widehat{m}_k)$ commute. 
	\end{prop}
	\begin{proof} One can easily check that conditions (b),(c), and (d) are equivalent to each others using equation (\ref{eqn:tau_w_a}). Moreover, a simple cross-ratio computation gives that $\tau_{k,j}=1$ if and only if the top and bottom vertices of the octahedron $o_k$ coincide. (See Propositions 4.13 and 4.14 in \cite{KKY}.) 
	\begin{figure}[H]
		\centering
		\scalebox{0.91}{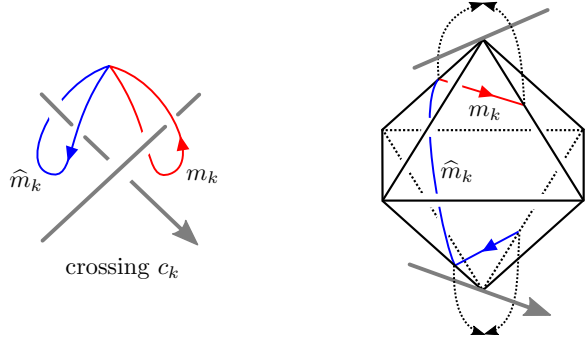}
		\caption{Wirtinger generators around a crossing $c_k$.}
		\label{fig:wirtinger}
	\end{figure}
	For condition (e) let us consider the Wirtinger generators $m_k$ and $\widehat{m}_k$ as in Figure \ref{fig:wirtinger}. As $m_k$ and $\widehat{m}_k$ wind the top and bottom vertices respectively as in Figure \ref{fig:wirtinger}, one can see that $\rho(m_k)$ and $\rho(\widehat{m}_k)$ fix the top and bottom vertices of a developing image of $o_k$, respectively. (See Remark 5.12 of \cite{KKY} for the details.) Since both $\rho(m_k)$ and $\rho(\widehat{m}_k)$ are parabolic elements, $\rho(m_k)$ and $\rho(\widehat{m}_k)$ commute if and only if the top and bottom vertices coincide, i.e., $o_k$ is pinched.
	\end{proof}
	
	We call the holonomy representation associated to a solution with pinched octahedra, or simply a pinched solution, a \emph{partially abelian representation} with respect to the diagram. We stress that the notion of partially abelian representations depends on the diagram. One can easily check that if every octahedron is pinched, then the solution gives an abelian representation. We also say ``a solution is pinched at a crossing $c_k$" to refer ``the  octahedron $o_k$ is pinched". Note that condition (e) is also equivalent to (e$'$) the $\rho$-images of the Wirtinger generators around $c_k$ commute.
	\begin{prop} \label{prop:ngon} Suppose that a region of $D$ has $n$ corner crossings. If a solution $w$ is pinched at $n-1$ octahedra among them, then it is also pinched at the last crossing.
	\end{prop} 
	\begin{proof} The proof directly follows from condition (d) of Proposition~\ref{prop:pinch} and the gluing equation (\ref{eqn:gluing}) for the region. (Alternatively, one may use Proposition~\ref{prop:pinch}(e).)
	\end{proof}
	
	\section{$R$-related diagrams}
	
	\subsection{Crossing change and diagram change}
	In this section, we propose a notion of R-relatedness of the knot diagrams by the following property : If two diagrams $D$ and $D'$ are R-related, then the knots $K$ and $K'$, given by $D$ and $D'$ respectively, have boundary parabolic representations with the same image group in $\psl$. To exclude the trivial case we assume that representations in this section are not abelian, equivalently solutions in this section are not pinched at every crossing.
	
	\begin{thm}\label{thm:main} Let region variables $w$ be a boundary parabolic solution for a diagram $D$. Suppose that the solution $w$ is pinched at crossings $\{c_k\; |\; k \in J\}$ for some index set $J$. Then the solution $w$ is also a boundary parabolic solution for a diagram $D^J$, which is obtained from $D$ by changing the crossings $\{c_k\; |\; k \in J\}$.
    \end{thm}
	\begin{proof} Let $\tau_{*,*}$(resp., $\tau^J_{*,*}$) be the $\tau$-values in equation (\ref{eqn:tau_w_a}) for the region variables $w$ with respect to the diagram $D$(resp., $D^J$). It is clear from equation (\ref{eqn:tau_w_a}) that $\tau_{k,*}=\tau^J_{k,*}$ for $k \notin J$. Also conditions (a) and (c) of Proposition \ref{prop:pinch} tell us that $\tau_{k,*}=\tau^J_{k,*}=1$ for $k \in J$.  Therefore the solution $w$ also satisfies the gluing equations for every region of $D^J$.
	\end{proof}
	 We say such two diagrams $D$ and $D^J$ in Theorem \ref{thm:main} are \emph{R-related}.  Let $K$(resp., $K^J$) be a knot represented by $D$(resp., $D^J$). The solution $w$ induces a representation for both the knot groups of $K$ and $K^J$, and we denote them by $\rho$ and $\rho^J$, respectively. One can describe $\rho^J$ by $\rho$ as follows. Let $m_i,m_j,$ and $m_l$ (resp., $m^J_i,m^J_j,$ and $m^J_l$) be Wirtinger generators around a crossing $c_k$($k\in J$) for the diagram $D$ (resp., $D^J$) as in Figure \ref{fig:D_J}. Then $\rho^J(m^J_i)=\rho(m_i)=\rho(m_l)$, $\rho^J(m^J_j)=\rho(m_j)$, and $\rho^J(m^J_l)=\rho(m_j)$. Note that we have $\rho(m_i)=\rho(m_l)$ and $\rho^J(m^J_j)=\rho^J(m^J_l)$.
	
	\begin{figure}[H]
		\centering
		\scalebox{1}{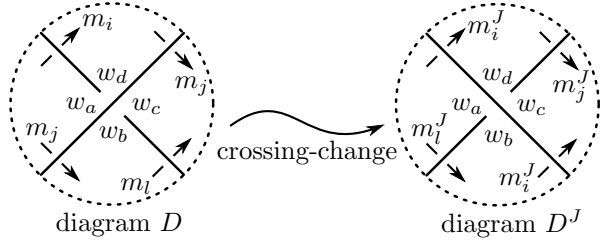}
		\caption{A crossing-change and Wirtinger generators.}
		\label{fig:D_J}
	\end{figure}
	Then it is clear that the image group of $\rho^J$ is the same as that of $\rho$. In particular, the complex volumes of $K$ and $K^J$ with respect to these representations are the same. 
	\begin{exam}[The knot $8_5$] In Section 7.2 of \cite{KKY}, they presented a pinched solution $w=(w_1,\cdots,w_{10})$ for a diagram $D$ of the knot $8_5$ as in Figure \ref{fig:8_15}(a), 	and argued that there is no others : 	
		\begin{equation*}
		\begin{array}{ccl}
		(w_1,\cdots,w_{10}) &=&	\left(-\dfrac{1}{p+q-pqr}+\dfrac{1}{p}+\dfrac{1}{q},\;-\dfrac{1}{p+q-pqr}+\dfrac{1}{p}+r\right.\\[11pt] 
		& &\quad \quad ,\; -\dfrac{1}{-p qr+p+q}+\dfrac{1}{p}+\dfrac{2}{q}-r,\; -\dfrac{1}{p+q-pqr}+\dfrac{1}{p}+\dfrac{1}{q}\\[11pt]
		& &\quad \quad ,\;-\dfrac{1}{p+q-pqr}+\dfrac{1}{q}+r,\; \dfrac{1}{p}+\dfrac{1}{q},\; r,\; \dfrac{1}{p}+\dfrac{1}{q}\\[11pt]
		& &\quad \quad \left.,\; -\dfrac{1}{p+q-pqr}+\dfrac{1}{q}+r,\; -\dfrac{1}{p+q-pqr}+\dfrac{1}{p}+r\right).
		\end{array}
		\end{equation*}
		One can check that the solution $w$ is pinched at the crossings $c_1$ and $c_2$ using condition (b) of Proposition \ref{prop:pinch}. Then by Theorem~\ref{thm:main} it also satisfies the gluing equations for $D^{\{1\}}$ and $D^{\{1,2\}}$, which are diagrams of the granny knot and the $T(3,4)$ torus knot, respectively. In particular, the complex volume($0+3.28987i$) of the knot $8_5$ with respect to the solution $w$ is the same to that of the granny knot which is twice the complex volume($0+1.64493i$) of the irreducible representation for the trefoil knot~\cite{cho_connceted_2015}. (Note that the trefoil knot has a unique irreducible boundary parabolic representation.)
        \begin{figure}[H]
			\centering
			\scalebox{1}{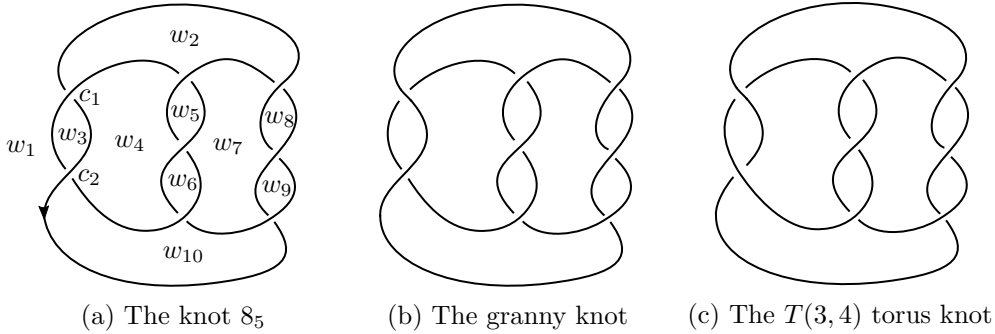}
			\caption{The $8_{5}$ knot diagram and $R$-related diagrams.}
			\label{fig:8_15}
		\end{figure}
	\end{exam}
	\begin{exam}[The knot $8_{18}$] Let us consider a diagram $D$ of the $8_{18}$ knot and assign region variables $w_1,\cdots,w_{10}$ to $D$ as in Figure \ref{fig:8_18}. 
		\begin{figure}[H]
			\centering
			\scalebox{1}{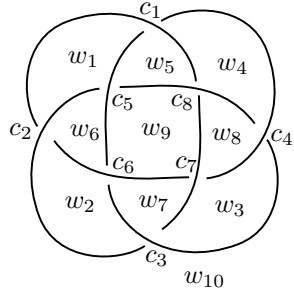}
			\caption{The $8_{18}$ knot diagram.}
			\label{fig:8_18}
		\end{figure}
		We first investigate the possibilities of crossings to be pinched. Suppose there is a solution $w$ pinched at the crossing $c_6$. By Theorem \ref{thm:main}, $w$ is a solution for $D^{\{6\}}$, which is a digram of the trefoil knot with a kink at the crossing $c_8$. Since Wirtinger generators around $c_8$ commute, the solution $w$ should be also pinched at $c_8$ by Proposition \ref{prop:pinch}(e). Under this condition one can compute a representation $\rho$ using the Wirtinger presentation. (We use Mathematica for the actual computation.) Also one can obtain the solution $w$ from the representation $\rho$ through \cite{cho_optimistic_2016} :
	\begin{equation*}
	\begin{array}{rl}
	w=&\big(p-q r+q,\ p-q r+q,\ p r+p-q r,\ p r+p-q r,\ p r-q r+q, \\[2pt]
	& p-q r+2 q,\ p r+p+q,\ 2 p r+p-q r,\ p r+q,p-q r\big).
	\end{array}
	\end{equation*} 
	We note that $w$ is pinched only at the crossings $c_6$ and $c_8$.
	Using the similar argument, we also obtain a solution $w'$ which is pinched at the crossings $c_2$ and $c_4$ :
	\begin{equation*}
	\begin{array}{rl}
	w'=&\big(p-q r+q,\ p r-2 q r+q,\ 4 p r-p-3 q r+q,\ p r+p-q r,\ p r-q r+q,\\[2pt]
	&p r-q r+q,\ 3 p r-p-2 q r+q,\ 3 p r-p-2 q r+q,\ 2 p r-p-q r+q,\ p-q r\big).
	\end{array}
	\end{equation*}
	One can check through Proposition \ref{prop:ngon} that other possibilities result in a solution pinched at every crossing or a solution in symmetry with $w$ or $w'$. Hence $w$ and $w'$ are the only pinched solutions for the diagram $D$.

	 Since both $D^{\{6,8\}}$ and $D^{\{2,4\}}$ represent the trefoil knot, we conclude that the knot $8_{18}$ has a boundary parabolic representation whose image is the modular group $\mathrm{PSL}(2,\Zbb)$, which is the image of the irreducible representation of the trefoil knot. 
	\end{exam}
	\begin{thm}\label{thm:main2} Let $w$ be a boundary parabolic solution for a diagram $D$. Suppose that the solution $w$ is pinched at a crossing $c_k$. Let $D'$ be a diagram obtained from $D$ by replacing $c_k$ by the standard diagram of a rational tangle $[2 n_1,\cdots,2 n_{k-1},2 n_{k}+1]$, $n_i \in \Zbb$. Then there is a boundary parabolic solution $w'$ for a diagram $D'$ such that $w'$ is pinched at every crossing in the tangle and coincide with $w$ on the outside of the tangle.
	\end{thm}
	\begin{proof} Let us denote the region variables around the crossing $c_k$ by $w_a,w_b,w_c,$ and $w_d$ as in Figure \ref{fig:rational_tangle}. Then we have $w_a-w_b+w_c-w_d=0$ from Proposition \ref{prop:pinch}(a). We will prove the theorem by induction on $k$. For the case $k=1$, we replace the crossing $c_k$ by a rational tangle $[2n_1+1]$. Compare with the diagram $D$, there are even number of new regions of $D'$. We define region variable $w'$ by assigning $w_c$ and $w_a$ alternately to these new regions and leave $w$ for other  unchanged regions. See Figure \ref{fig:rational_tangle}. It is clear that $w'$ is pinched at every crossing in the tangle, since we have $w_a-w_b+w_c-w_d=0$ at each crossing. Also one can check that $w'$ satisfies the gluing equation for every region of $D'$ by Proposition \ref{prop:pinch}(d).
		
	For $k>1$, the number of regions of $D'$ increases by an even number as $k$ increases by $1$. We define $w'$ by assigning $w_b$ and $w_d$(resp., $w_c$ and $w_a$) alternately to the newly created regions if $k$ increase to an even(resp., odd) number as in Figure \ref{fig:rational_tangle}. Then one can check that $w'$ is a desired solution.
	\begin{figure}[H]
		\centering
		\scalebox{1}{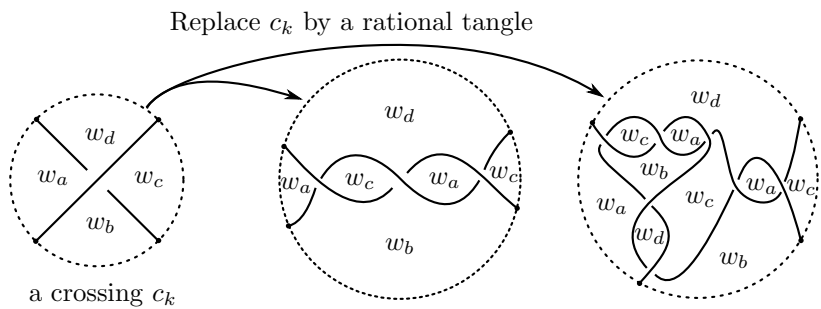}
		\caption{Rational tangles $[3]$ and $[2,-2,3]$.}
		\label{fig:rational_tangle}
	\end{figure}
	\end{proof}
	\subsection{R-related diagrams from connected sum}
	Now we will give some examples of R-related diagrams using connected sum. 
	Let $D$(resp., $D'$) be a diagram of a knot $K$(resp., $K'$) and let $\rho$(resp., $\rho'$) be a boundary parabolic representation of the knot group of $K$(resp., $K'$). One can construct $1$-parameter family of boundary parabolic representations for $K \# K'$ as follows. Let $A$ and $A'$ be arcs of $D$ and $D'$ respectively which are to be cut for the connected sum $D \# D'$. We may assume that both $\rho(m_A)$ and $\rho'(m_{A'})$ are $\left(\begin{array}{cc} 1 & 1 \\ 0& 1 \end{array}\right)$ by conjugating $\rho$ and $\rho'$ appropriately where $m_A$(resp., $m_{A'}$) is the Wirtinger generator winding the arc $A$(resp., $A'$). Then for any $r \in \Cbb$ we define a representation $\rho \#_r \rho'$ for $D \# D'$ by assigning $\rho$ to the Wirtinger generators winding arcs of $D$ and assigning $\left(\begin{array}{cc} 1 & r \\ 0& 1 \end{array}\right)\rho'\left(\begin{array}{cc} 1 & -r \\ 0& 1 \end{array}\right)$ to the Wirtinger generators of $D'$. (This construction is also described in \cite{cho_connceted_2015}.)
	
	Now choose an arc $B$ of $D$ and an arc $B'$ of $D'$ such that they are parts of a common region in $D \# D'$. Suppose both $\rho(m_B)$ and $\rho'(m_{B'})$ do not fix $\infty$ where $m_B$(resp., $m_{B'}$) is the Wirtinger generator winding the arc $B$(resp., $B'$). Let us choose $r:= \textrm{Fix}(\rho(m_B))-\textrm{Fix}(\rho'(m_{B'}))$. Then the $\rho\#_r\rho'$ image of $m_B$ and $m_{B'}$ commute since they are parabolic elements having a common fixed point. Therefore, applying Reidemeister second move for the arcs $B$ and $B'$ in the common region, we obtain two pinched crossings, satisfying  Proposition~\ref{prop:pinch}(e).
	\begin{exam}[The granny knot] Let $D$ and $D'$ be diagrams of the trefoil knot, and $\rho$ and $\rho'$ be representations described as in Figure \ref{fig:granny}(a). Also we choose the arcs $A,A',B,$ and $B'$ as in Figure \ref{fig:granny}(a). 
	\begin{figure}[H]
		\centering
		\scalebox{0.8}{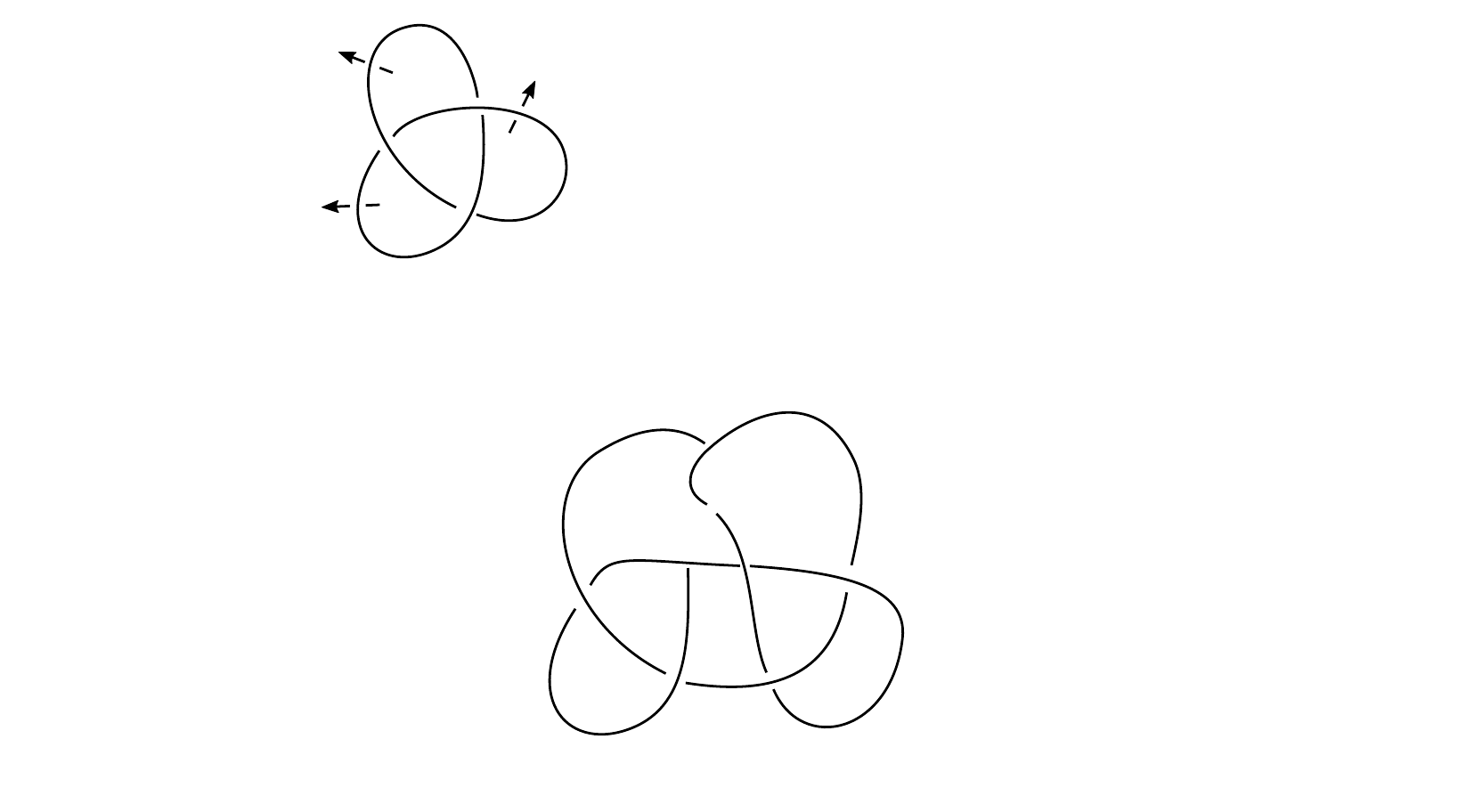}
		\caption{R-related diagrams : the granny knot, $8_{21}$, and $8_{15}$.}
		\label{fig:granny}
	\end{figure}
	Then we have $r= \textrm{Fix}(\rho(m_B))-\textrm{Fix}(\rho'(m_{B'}))=-1-0=-1$ and hence we obtain the irreducible representation $\rho \#_{-1} \rho'$ for $D \# D'$. Now apply Reidemeister second move for the arcs $B$ and $B'$ in $D \# D'$. Since the $\rho \#_{-1} \rho'$ images of $m_B$ and $m_{B'}$ commute, $\rho \#_{-1} \rho'$ is also a representation for a diagram obtained from $D \# D'$ by changing a crossing created by the Reidemeister move. This results in the knots $8_{21}$ and $8_{15}$ depending on the crossing-change as in Figure \ref{fig:granny}.	Hence each of the knots $8_{21}$ and $8_{15}$ has a representation whose image group is the same as the image of $\rho \#_{-1} \rho'$ of the granny knot.

		One can use different arcs $B$ and $B'$ as in Figure \ref{fig:granny2} which results in diagrams of the knots $8_{19}$ and $8_5$. We note that other choices for $B$ and $B'$ does not give a new one.
		\begin{figure}[H]
			\centering
			\scalebox{0.8}{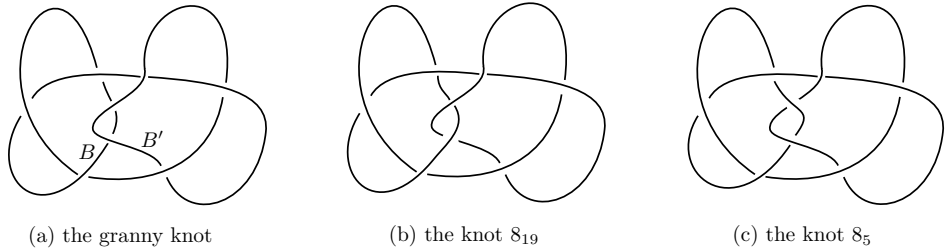}
			\caption{R-related diagrams : the granny knot, $8_{19}$, and $8_{5}$.}
			\label{fig:granny2}
		\end{figure}
	\end{exam}
	\begin{exam}[The square knot] We can apply the same argument to the square knot and obtain the knots $8_{20}$ and $8_{10}$ as in Figure \ref{fig:square}. We check that these knots are only knots obtained from the square knot diagram. Again, we can conclude that each of the knots $8_{20}$ and $8_{10}$ have a representation whose image is the same as a  that of the square knot.
	\begin{figure}[H]
		\centering
		\scalebox{0.8}{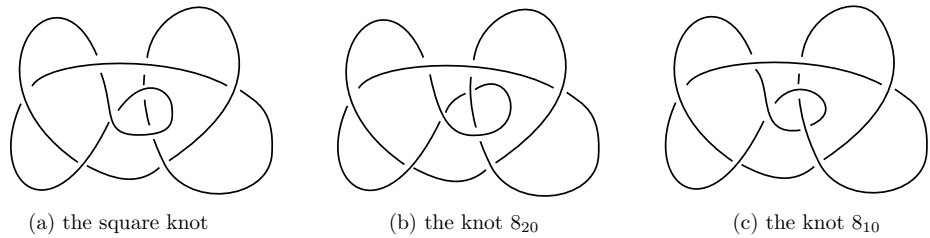}
		\caption{R-related diagrams : the square knot, $8_{20}$, and $8_{10}$.}
		\label{fig:square}
	\end{figure}
	\end{exam}
	The discussions in this section suggest implicitly a hierarchy on the set of knots. If two knots share a R-related digram, in general one is ``smaller" than the other in a certain sense as the discussions in this section indicate, i.e., a representation of a smaller knot essentially appears as a pinched representation of the other bigger knot. This may define a kind of order or a hierarchy on the set of knots This also suggest a strong relation with the knot group epimorphism problem, even though this hierarchy looks weaker than the partial order defined by the knot group epimorphism \cite{KS_2005}. As we saw in the examples of 8 crossing knots, all these knots obtained from granny and square knots by crossing changes are known to have an epimorphism to the trefoil knot, and in fact these are the only such knots with up to 8 crossings. We hope to investigate this ``hierarchy" and the relationship with epimorphism problem more systematically in future papers.
\bibliographystyle{unsrt}

\end{document}